\documentclass{amsart}
 \newtheorem{thm}{Theorem}[section]
 \newtheorem{cor}[thm]{Corollary}
 \newtheorem{lem}[thm]{Lemma}

\begin{document}

%-------------------------------------------------------------------------
% editorial commands: to be inserted by the editorial office
%
%\firstpage{1} \volume{228} \Copyrightyear{2004} \DOI{003-0001}
%\seriesextraline{This is the Concrete Title of this Book\br H.E. R and S.T.C. W, Eds.}
%\firstpage{1}

\title[Which quartic polynomials have a hyperbolic antiderivative?]
 {Which quartic polynomials have a hyperbolic antiderivative?}

%----------Author 1
\author[Rajesh Pereira]{Rajesh Pereira}

\address{%
Department of Mathematics and Statistics\\
University of Guelph\\
Guelph ON. N1G 2W1\\
CANADA}

\email{pereirar@uoguelph.ca}

\thanks{This work supported by NSERC and the hospitality of the Mittag-Leffler Institute. I would like to thank Profs. Vladimir Kostov and Alan Horwitz for their suggestions which led to this second version of this ArXiv paper.}
%----------Author 2



\date{June 29, 2018}
\dedicatory{In Memory of Serguei Shimorin}
%%% ----------------------------------------------------------------------

\begin{abstract}
Every linear, quadratic or cubic polynomial having all real zeros is the derivative of a polynomial having all real zeros.  The statement is false for higher degree polynomials.  In particular, not every fourth degree polynomial with real zeros is the derivative of a polynomial having all real zeros.  We derive a necessary and sufficient condition for a quartic polynomial to be the derivative of a polynomial having all real zeros.  This condition is a single quadratic form inequality involving the zeros of the quartic polynomial.
\end{abstract}

%%% ----------------------------------------------------------------------
\maketitle
%%% ----------------------------------------------------------------------
%\tableofcontents
\section{Introduction}

The relationship between the zeros of a polynomial and those of its derivative has been of significant interest to mathematicians for at least three centuries.  Serguei Shimorin has worked in this area \cite{khavinson2011borcea}.  In this paper, we will study polynomials having all of their zeros on the real line; these are sometimes called hyperbolic polynomials.  It is a simple consequence of Rolle's theorem that the derivative of a hyperbolic polynomial is a hyperbolic polynomial.

 The converse to this is false.  A hyperbolic polynomial of degree three or less always has a hyperbolic antiderivative.  However for $n\geq 4$, there are $n$th degree hyperbolic polynomials which have no hyperbolic antiderivatives at all.  The example $p(x)=(x-1)^2(x-4)^2$ was given in \cite{bilodeau1991generating}.

It would be desirable to have a systematic test for this.  Suppose $Q(z)$ is an $(n+1)$th degree monic hyperbolic polynomial with zeros $z_1\ge z_2\ge z_3 \ge ... \ge z_n\ge z_{n+1}$. Then $Q(z)$ is nonnegative on $[z_{2j+1},z_{2j}]$ and nonpositive on $[z_{2j},z_{2j-1}]$ for all $j:1\le j\le \frac{n}{2}$.  Now let  $w_1\ge w_2\ge w_3\ge ...\ge w_n$ be the zeros of $Q'(z)$.  By Rolle's theorem $z_{j+1}\le w_j \le z_j$ and hence $Q(w_{2j})\geq 0$ and $Q(w_{2j-1})\leq 0$ for all $j$.  Conversely if $Q$ is a real monic $(n+1)$th degree polynomial and there exists $w_1\ge w_2\ge w_3\ge ...\ge w_n$ with $Q'(w_j)=0$, $Q(w_{2j})\geq 0$ and $Q(w_{2j-1})\leq 0$ for all $j$, then $Q$ is hyperbolic by the Intermediate Value theorem.

 This is a nice characterization, however it is in terms of $Q(z)$.  If we start with the polynomial $p(z)$ and find an antiderivative, there is no guarantee that we will get the particular $Q(x)$ which has all of its zeros real, we will instead get $P(z)=Q(z)+c$ for some arbitrary real $c$.  This will shift everything by $c$ which gives us the criterion of Souroujon and Stoyanov.

\begin{lem} \cite{SoSt}
Let $\{ w_k\}_{k=1}^{n}$ be real numbers with $w_1\ge w_2\ge...\ge w_{n-1}\ge w_n$.  Let $p(x)=\prod_{k=1}^{n}(x-w_k)$ and let $P(x)$ be any antiderivative of $p(x)$, then there exists $c\in \mathbb{R}$ such that $P(x)-c$ has all zeros real if and only if $\max \{P(w_k): k$ odd $\} \le \min \{P(w_k): k$ even $\}$ in which case we can take any choice of $c$ such that $\max \{P(w_k): k$ odd $\} \le c \le \min \{P(w_k): k$ even $\}$. \end{lem}

We can restate this Lemma in a more convenient form.

\begin{lem} \label{firstlem} Let $\{ w_k\}_{k=1}^{n}$ be real numbers with $w_1\ge w_2\ge...\ge w_{n-1}\ge w_n$.  Let $p(x)=\prod_{k=1}^{n}(x-w_k)$ and let $P(x)$ be any antiderivative of $p(x)$, then there exists $c\in \mathbb{R}$ such that $P(x)-c$ has all zeros real if and only if $P(w_j)\geq P(w_k)$ whenever $j$ is even and $k$ is odd and $\vert j-k\vert\geq 3$.  \end{lem}

We note that if $j-k=1$, then $p(x)>0$ on the interval $(w_k, w_j)$ and if $k-j=1$, then $p(x)<0$ on the interval $(w_j, w_k)$; therefore in both cases we automatically get $P(w_j)\geq P(w_k)$ which is why we can drop these as conditions in Lemma \ref{firstlem}.  (Interestingly, while this fact will not play a role in this paper, these inequalities are the only conditions on the ordered set $\{ P(w_j)\}$ for arbitrary hyperbolic polynomials $P$.  See \cite{davis1956problem} for the exact statement, proof and discussion of this fact.)

\section{Quartic Polynomials}

 A simple induction shows that the number of inequalities in Lemma \ref{firstlem} is \newline $\lfloor (\frac{n}{2}-1)^2\rfloor $ when $n \geq 2$.
In particular, we see that for fourth degree polynomials the existence of a hyperbolic antiderivative essentially is equivalent to a single condition.  We state this special case of Lemma \ref{firstlem}.

\begin{cor} Let $\{ w_k\}_{k=1}^{n}$ be real numbers with $w_1\ge w_2\ge w_3\ge w_4$.  Let $p(x)=\prod_{k=1}^{4}(x-w_k)$ and let $P(x)$ be any antiderivative of $p(x)$, then there exists $c\in \mathbb{R}$ such that $P(x)-c$ has all zeros real if and only if $P(w_4)\geq P(w_1)$. \end{cor}

We note that if $a$ and $b$ are real numbers with $a\neq 0$ then $p(ax+b)$ is a hyperbolic polynomial with a hyperbolic antiderivative if and only if $p(x)$ is.  We may therefore apply the transformation $ax+b$ which maps $w_1$ to $1$ and $w_4$ to $-1$ and consider quartic polynomials having $1$, $-1$, $s$ and $t$ as zeros with $s,t\in [-1,1]$.  In this case, we get a very simple condition in terms of the zeros of $p$.

\begin{thm} Let $s,t\in [-1,1]$ and let $p(x)=(x-1)(x-s)(x-t)(x+1)$.  Then $p(x)$ has a hyperbolic antiderivative if and only if $st\geq -\frac{1}{5}$. \end{thm}

\begin{proof} Since $p(x)=x^4-(s+t)x^3+(st-1)x^2+(s+t)x-st$, we get $60P(x)=12x^5-15(s+t)x^4+20(st-1)x^3+30(s+t)x^2-60stx$ where $P(x)$ is an antiderivative of $p(x)$.   Now $60(P(1)-P(-1))=2(12+20(st-1)-60st)=8(3+5(st-1)-15st)=8(-2-10st)=-16(1+5st)$; which means $p$ has a hyperbolic antiderivative if and only if $st\geq -\frac{1}{5}$.  \end{proof}

We note that the mapping $\frac{w_1 - w_4}{2}x-\frac{w_1+w_4}{2}$ maps the numbers $-1,s,t,1$ (where $s=\frac{2w_2-w_1-w_4}{w_1-w_4}$ and $t=\frac{2w_3-w_1-w_4}{w_1-w_4}$) to $w_1, w_2, w_3, w_4$.  The inequality $st\geq -\frac{1}{5}$ is equivalent to $5(2w_2-w_1-w_4)(2w_3-w_1-w_4)+(w_1-w_4)^2\geq 0$.  After some algebra, we can restate this condition as follows:

\begin{thm} Let $\{w_i\}_{i=1}^4$ be real numbers with $w_1\ge w_2\ge w_3\ge w_4$ and let $p(x)=(x-w_1)(x-w_2)(x-w_3)(x-w_4)$.  Then $p(x)$ has a hyperbolic antiderivative if and only if $w^t Aw\geq 0$ where $w=(w_1,w_2,w_3,w_4)$ and where

\[
A=
\begin{bmatrix}
6 & -5 & -5 &  4 \\
-5 & 0 & 10 & -5 \\
-5 & 10 & 0 & -5\\
 4 & -5 & -5 & 6
\end{bmatrix} \ \ .
\]\end{thm}

We can also reformulate this result in terms of the gaps between the zeros.  Let $g_j=w_j -w_{j+1}$ for $j=1,2,3$.  Then $5(2w_2-w_1-w_4)(2w_3-w_1-w_4)+(w_1-w_4)^2=5(-g_1+g_2+g_3)(-g_1-g_2+g_3)+(g_1+g_2+g_3)^2=5(g_3-g_1)^2-5g_2^2+(g_1+g_2+g_3)^2$.  This gives us the following result.

\begin{thm} Let $\{w_i\}_{i=1}^4$ be real numbers with $w_1\ge w_2\ge w_3\ge w_4$ and let $p(x)=(x-w_1)(x-w_2)(x-w_3)(x-w_4)$.  Then $p(x)$ has a hyperbolic antiderivative if and only if $v^t Bv\geq 0$ where $v=(w_1-w_2,w_2-w_3,w_3-w_4)$ is the vector of distances between adjacent zeros of $p$ and where

\[
B=
\begin{bmatrix}
6 & 1 &  -4 \\
1 & -4 & 1 \\
-4 & 1 & 6
\end{bmatrix} \ \ .
\]  \end{thm}

This suggests the problem of finding the characterization of the zero sets of higher degree polynomials which have hyperbolic antiderivatives.  It is clear that the characterization will be a set of homogeneous polynomial inequalities of the form $S_n=\{ w_1\ge w_2\ge...\ge w_n: p_{n,k}(w_1,w_2,...,w_n)\geq 0;$\  $1\le k \le \lfloor (\frac{n}{2}-1)^2\rfloor  \}$.  The degrees of these polynomials are less than or equal to $n-2$.

It is interesting to note that a related problem has been fairly well studied.  A polynomial $p$ is said to be very hyperbolic if it is the $n^{th}$ derivative of a hyperbolic polynomial for all natural numbers $n$.  V.~P.~Kostov has extensively studied the geometry of the sets of very hyperbolic polynomials; see the references \cite{dimitrov2011sharp,kostov2011topics,kostov2005very} for more details.

\bibliographystyle{plain}

\bibliography{shimbib}

% ------------------------------------------------------------------------
\end{document}